\documentclass[
final
]{dmtcs-episciences}

\usepackage{amsmath}
\usepackage{amsfonts}
\usepackage{amsthm}

\usepackage{algpseudocode,algorithm,algorithmicx}

\algrenewcommand\algorithmicrequire{\textbf{Precondition:}}
\algrenewcommand\algorithmicensure{\textbf{Postcondition:}}

\newtheorem{theorem}{Theorem}
\newtheorem{proposition}{Proposition}[theorem]
\newtheorem{corollary}{Corollary}[theorem]
\newtheorem{lemma}[theorem]{Lemma}

\newtheorem*{problem}{Problem}
\newtheorem{fact}[theorem]{Fact}

\theoremstyle{definition}
\newtheorem{definition}{Definition}
 
\theoremstyle{remark}

\newcommand{\ZZ}{\mathbb{Z}}
\newcommand{\NN}{\mathbb{N}}
\newcommand{\PP}{\mathbb{P}}
\newcommand{\RR}{\mathbb{R}}
\newcommand{\EE}{\mathbb{E}}

\newcommand{\cC}{\mathcal{C}}
\newcommand{\goesto}{\rightarrow}

\title{On Weakly Distinguishing Graph Polynomials}

\author{
  Johann A. Makowsky\affiliationmark{1}
  \and
  Vsevolod Rakita\affiliationmark{2}
}
\date{\today}
\affiliation{
  Technion - Israel Institute of Technology, Department of Computer Science\\
  Technion - Israel Institute of Technology, Department of Mathematics}

\received{2018-11-1}

\revised{2019-3-05}

\accepted{2019-3-25}

\begin{document}

\publicationdetails{21}{2019}{1}{4}{4949}
\maketitle

\begin{abstract}
A univariate graph polynomial \begin{math} P(G;X) \end{math} is weakly distinguishing if for almost all finite graphs 
\begin{math}G\end{math} there is a finite graph \begin{math}H\end{math} with \begin{math}P(G;X)=P(H;X)\end{math}. 
We show that the clique polynomial and the independence 
polynomial are weakly distinguishing. Furthermore, 
we show that generating functions of induced subgraphs with property \begin{math}C\end{math} are weakly distinguishing 
provided that \begin{math}C\end{math} is of bounded degeneracy or treewidth. The same holds for the harmonious chromatic polynomial. 
\end{abstract}

\section{Introduction and Outline}
Throughout this paper we consider only simple (i.e. finite, undirected loopless graphs without parallel edges), vertex labelled graphs.
Let \begin{math}P\end{math} be a graph polynomial.
A graph \begin{math}G\end{math} is {\em \begin{math}P\end{math}-unique} if every graph \begin{math}H\end{math} with \begin{math}P(G;X)=P(H;X)\end{math} is isomorphic to \begin{math}G\end{math}.
A graph {\em \begin{math}H\end{math} is a \begin{math}P\end{math}-mate} of \begin{math}G\end{math} 
if \begin{math}P(G;X)=P(H;X)\end{math} but \begin{math}H\end{math} is not isomorphic to \begin{math}G\end{math}. 
In \cite{noy2003graphs} \begin{math}P\end{math}-unique graphs are studied for the Tutte polynomial \begin{math}T(G;X,Y)\end{math}, 
the chromatic polynomial \begin{math}\chi(G;X)\end{math}, the matching polynomial \begin{math}m(G;X)\end{math} and the
characteristic polynomial \begin{math}char(P;X)\end{math}.

A statement holds for almost all graphs if the proportion of graphs of order \begin{math}n\end{math} 
for which it holds, tends to \begin{math}1\end{math}, when
\begin{math}n\end{math} tends to infinity.
A graph polynomial \begin{math}P\end{math} is 
{\em almost complete} if almost all graphs \begin{math}G\end{math} are \begin{math}P\end{math}-unique, and it is
{\em weakly distinguishing} if almost all graphs \begin{math}G\end{math} have a \begin{math}P\end{math}-mate.
In \cite{bollobas2000contraction} it is conjectured that almost all graphs are \begin{math}\chi\end{math}-unique and  \begin{math}T\end{math}-unique,
in other words, both \begin{math}\chi(G;X)\end{math} and \begin{math}T(G;X,Y)\end{math} are almost complete.
There are plenty of trivial graph polynomials which are weakly distinguishing, like 
\begin{math}X^{|V(G)|}\end{math} or \begin{math}X^{|E(G)|}\end{math}. However, one might expect that the prominent graph polynomials  
from the literature are not weakly distinguishing.
Here we show that various non-trivial graph polynomials are still weakly distinguishing.


The degree polynomial \begin{math}Deg(G;x)\end{math} of a graph \begin{math}G\end{math} 
is the generating function of the degree sequence of $G$.
A graph \begin{math}G\end{math} is \begin{math}Deg\end{math}-unique, also called in the literature
a {\em unigraph},
if it is determined by its degree sequence. 
An updated discussion on how to
recognize unigraphs can be found in \cite{borri2009recognition}. 

A simple counting argument gives:
\begin{theorem}
\label{theorem deg}
Almost all graphs \begin{math}G\end{math} have a \begin{math}Deg\end{math}-mate.
\end{theorem}


The {\em Independence and  Clique polynomials} of a graph \begin{math}G=(V(G),E(G))\end{math} contain much information about $G$. 
Both were first studied in \cite{hoede1994clique}. For a more recent survey on the independence polynomial see \cite{levit2005independence}.

\begin{theorem}
\label{theorem clique ind}
The independence and clique polynomials are weakly distinguishing.
\end{theorem}

The proof uses estimates for the independence number \begin{math}\alpha(G)\end{math} and the clique number \begin{math}\omega(G)\end{math}
for random graphs, (see \cite{bollobas1976cliques} and \cite{frieze1990independence}) together with a counting argument.\\

This theorem can be generalized:
\begin{definition}
\label{definition ind function}
 Let \begin{math}\mathcal{C}\end{math} be a graph property.
We say that a function \begin{math}f:\mathbb{N} \rightarrow \mathbb{N}\end{math} is an 
{\em independence (clique) function for \begin{math}\mathcal{C}\end{math}}
if for every graph \begin{math}G\in \mathcal{C}\end{math}, the graph \begin{math}G\end{math}
 has an independent set (clique) of size \begin{math}f(|V(G)|)\end{math}.
\end{definition}
 
Denote by \begin{math}\hat{\mathcal{C}}\end{math} the class of complement graphs \begin{math}\bar{G}\end{math} of graphs
\begin{math}G \in \mathcal{C}\end{math}, and \begin{math}P_{\mathcal{C}}(G;X) = \sum_{A \subset V(G): G[A] \in \mathcal{C}} X^{|A|}\end{math}, and
\begin{math}P_{\hat{\mathcal{C}}}(G;X) = \sum_{A \subset V(G): G[A] \in \hat{\mathcal{C}}} X^{|A|}\end{math}.

\begin{theorem}
\label{theorem classes}
Let \begin{math}Q\end{math}  be a graph property that has an independence or a clique function 
\begin{math}f\end{math}  that satisfies that for all \begin{math}n\in \NN\end{math} ,
\begin{math}f(n)\geq n/a\end{math}  for some fixed \begin{math}a\in \NN\end{math} . Then \begin{math}P_Q\end{math}  is weakly distinguishing.
\end{theorem}

This applies to the following cases:

\begin{itemize}
\item A graph \begin{math}G\end{math} is \begin{math}k\end{math}-degenerate if every induced subgraph of \begin{math}G\end{math} has a vertex of degree at most \begin{math}k\end{math}.
It is easy to see that
every \begin{math}k\end{math}-degenerate graph \begin{math}G\end{math} of order \begin{math}n\end{math} has an independent set of size \begin{math}{\left\lceil \frac{n}{k+1} \right\rceil}\end{math}.
\item Among the \begin{math}k\end{math}-degenerate graphs we find the graphs of treewidth at most \begin{math}k\end{math}, graphs of degree at most \begin{math}k\end{math}, 
and planar graphs.
\item A \begin{math}k\end{math}-colourable graphs \begin{math}G\end{math} has an independent set of size at least \begin{math}{\left\lceil \frac{n}{k} \right\rceil}\end{math}.
\item Let \begin{math}\mathcal{C}\end{math} be a graph property. A function \begin{math}\gamma: V(G) \rightarrow [k]\end{math} is a {\em \begin{math}\mathcal{C}\end{math}-colouring} if every
color class induces a graph in \begin{math}\mathcal{C}\end{math}.
Such coloring were studied in \cite{Gutman83Generalizations}.
If we assume that \begin{math}\mathcal{C}\end{math} has an independence (clique) function \begin{math}g(n)\end{math}, then
the graphs which are \begin{math}\mathcal{C}\end{math}-colorable with at most \begin{math}k\end{math} colors have an independence (clique) function
\begin{math}f(n) = {\left\lceil \frac{g(n)}{k} \right\rceil}\end{math}.
\end{itemize}

Therefore, for \begin{math}\mathcal{C}\end{math} one of the properties above, the graph polynomials 
\begin{math}P_{\mathcal{C}}(G;X)\end{math}
are all weakly distinguishing.

A harmonious colouring of \begin{math}G\end{math} with at most \begin{math}k\end{math} colors is a proper colouring of \begin{math}G\end{math} such that every pair of colors
occurs at most once along an edge.
Let \begin{math}\chi_{harm}(G;k)\end{math} count the number of harmonious colourings of \begin{math}G\end{math}.
It was observed in \cite{makowsky2006polynomial} that \begin{math}\chi_{harm}(G;k)\end{math} is a polynomial in \begin{math}k\end{math}.

\begin{theorem}
\label{theorem harmonious}
Almost all graphs \begin{math}G\end{math} have a \begin{math}\chi_{harm}\end{math}-mate.
\end{theorem} 
The status of \begin{math}P\end{math}-uniqueness remains open for \begin{math}T(G;X,Y), \chi(G;X), m(G;X)\end{math} and \begin{math}char(G;X)\end{math}.

\section{Preliminaries}
 
Let \begin{math}G=(V,E)\end{math} be a graph. 
Denote by \begin{math}\mathcal{G}(n)\end{math} the set of all non-isomorphic graphs with \begin{math}n\end{math} vertices, 
and by \begin{math}\mathcal{G}\end{math} the set of all non-isomorphic graphs.
 
\begin{fact}[\cite{harary2014graphical}]
\begin{math}|\mathcal{G}(n)|\approx \frac{2^{{n \choose 2}}}{n!} \end{math} for a sufficiently large \begin{math}n\end{math}.
\end{fact}

Let \begin{math}P:\mathcal{G} \goesto \ZZ[x]\end{math} be a graph polynomial. 
For a graph \begin{math}G\end{math}, we say  two non-isomorphic graphs \begin{math}G\end{math} and \begin{math}H\end{math} are called \begin{math}P\end{math}-mates if \begin{math}P(G)=P(H)\end{math}. \\
Denote by \begin{math}U_P(n)\end{math} the set of \begin{math}P\end{math} unique graphs with \begin{math}n\end{math} vertices,
 by \begin{math}\beta_P(n)\end{math} the number of polynomials in \begin{math}\ZZ[x]\end{math} 
 such that there  is a graph of order less or equal to-\begin{math}n\end{math} that maps to that polynomial, 
 and by \begin{math}\beta_{P,\mathcal{C}}(n)\end{math} the number of polynomials in \begin{math}\ZZ[x]\end{math} 
 such that there  is a graph of order less or equal to-\begin{math}n\end{math} in \begin{math}\mathcal{C}\end{math} that maps to that polynomial.\\
We denote by \begin{math}K_n\end{math} the clique of size \begin{math}n\end{math}, 
and by \begin{math}I_n\end{math} the edgeless graph of size \begin{math}n\end{math}.\\
Let \begin{math}G\end{math} be a graph, and \begin{math}A\subseteq V(G)\end{math}. The induced subgraph of \begin{math}A\end{math} in \begin{math}G\end{math}, 
denoted \begin{math}G[A]\end{math}, is the graph with vertex set \begin{math}A\end{math}, and for \begin{math}v,u \in A\end{math}, \begin{math}(u,v)\in E(G[A])\end{math} iff \begin{math}(u,v) \in E(G)\end{math}.

\begin{definition}
For a graph polynomial \begin{math}P\end{math}, we say \begin{math}P\end{math} is weakly distinguishing if 
\begin{math}lim_{n \goesto \infty} \frac{|U_P(n)|}{|\mathcal{G}(n)|}=0\end{math}. 
For a family of graphs \begin{math}\mathcal{C}\end{math} we say that \begin{math}P\end{math} is weakly distinguishing on 
\begin{math}\mathcal{C}\end{math} if \begin{math}lim_{n \goesto \infty} \frac{|U_P(n)\cap \mathcal{C}|}{|\mathcal{G}(n) \cap \mathcal{C}|}=0\end{math}
\end{definition}
We wish to consider a particular type of graph polynomials:
\begin{definition}
Let \begin{math}Q\end{math} be a graph property. For all graphs \begin{math}G\end{math}, define \begin{math}P_Q(G;x)=\sum_{A \subset V(G): G[A] \in Q} X^{|A|}\end{math}.
\end{definition}

\section{The Degree Polynomial}
\begin{definition}
For a graph \begin{math}G=(V,E)\end{math} of order \begin{math}n\end{math} and \begin{math}v \in V\end{math}, denote by \begin{math}deg(v)\end{math} the degree of \begin{math}v\end{math}. 
Define the Degree polynomial of \begin{math}G\end{math} to be \begin{math}Deg(G,x)=\sum_{v\in V} x^{deg(v)}\end{math}.
\end{definition}
Note that the degree of a vertex is bounded above by \begin{math}n-1\end{math}, 
so the degree of the polynomial \begin{math}Deg(G,x)\end{math} is at most \begin{math}n-1\end{math}. For every \begin{math}0\leq i\leq n-1\end{math}, 
the coefficient of \begin{math}x^i\end{math} in \begin{math}Deg(G,x)\end{math} is an integer number between \begin{math}0\end{math} and \begin{math}n\end{math}. Thus, we get 
\begin{displaymath}\beta_{Deg}(n)\leq (n+1)^{n-1}\leq (n+1)^n\end{displaymath}

Now we are ready to prove theorem \ref{theorem deg}:\\
{\bf Theorem \ref{theorem deg}}: Almost all graphs \begin{math}G\end{math} have a \begin{math}Dg\end{math}-mate.\\
\begin{proof}
Let \begin{math}G=(V,E)\end{math} be a graph with \begin{math}|V(G)|=n\end{math}. 
 We now evaluate:
\begin{math}\\ \\
\begin{aligned}
\lim_{n\goesto \infty} \dfrac{U_{Deg}(n)}{|\mathcal{G}(n)|}\leq \lim_{n\goesto \infty} 
\dfrac{\beta_{Deg}(n)}{|\mathcal{G}(n)|}\leq \lim_{n\goesto \infty} \dfrac{n^n}{|\mathcal{G}(n)|} =
 \lim_{n\goesto \infty} \dfrac{(n+1)^n n!}{2^{n(n-1)/2}} 
\\ \leq \lim_{n\goesto \infty} \dfrac{(n+1)^n \cdot (n+1)^n}{2^{n(n-1)/2}}= \lim_{n\goesto \infty} \dfrac{(n+1)^{2n}}{2^{n(n-1)/2}}=0
\end{aligned}\\ \\
\end{math}
\end{proof}

\section{A General Method for Proving Graph Polynomials are Weakly Distinguishing}
We wish to apply the same idea used in proving the degree polynomial is weakly distinguishing
 to a large class of  graph polynomials. We start with some lemmas.
First, we show that if a graph polynomial \begin{math}P\end{math} is weakly distinguishing on a 
large subset of \begin{math}\mathcal{G}\end{math}, it is weakly distinguishing:
\begin{lemma}
\label{lemma 1}
Let \begin{math}P\end{math} be a graph polynomial and \begin{math}\mathcal{C}\end{math} a family of graphs such that \begin{math}\lim_{n\goesto \infty} |\mathcal{C}(n)|/|\mathcal{G}(n)| =1\end{math}. 
If \begin{math}\lim_{n\goesto \infty} |U_P(n) \cap \mathcal{C}|/|\mathcal{G}(n)|=0\end{math} then \begin{math}P\end{math} is weakly distinguishing.
\end{lemma}
\begin{proof}

\begin{displaymath}
\dfrac{|U_P(n)|}{|\mathcal{G}(n)|}=
\dfrac{|U_P(n)\cap \mathcal{C}|+|U_P(n)\cap (\mathcal{G}(n)-\mathcal{C})|}{|\mathcal{G}(n)|}=
\dfrac{|U_P(n)\cap \mathcal{C}|}{|\mathcal{G}(n)|}+\dfrac{|U_P(n)\cap (\mathcal{G}(n)-\mathcal{C})|}{|\mathcal{G}(n)|}
\end{displaymath}

When taking the limit, note that the left term in the sum converges to 0 by assumption, so it remains to evaluate:

\begin{displaymath}
\lim_{n \goesto \infty} \dfrac{|U_P(n)\cap (\mathcal{G}-\mathcal{C})|}{|\mathcal{G}(n)|} \leq
 \lim_{n \goesto \infty} \dfrac{|\mathcal{G}(n)-\mathcal{C}|}{|\mathcal{G}(n)|}=0
\end{displaymath}
\end{proof}

\begin{lemma}
\label{lemma 2}
Let \begin{math}f:\NN\goesto \RR\end{math}. If \begin{math}f(n)\leq (\log n)^{O(1)}\end{math} , then asymptotically 

\begin{displaymath}{n \choose f(n)}^{f(n)} \leq (\dfrac{n}{f(n)})^{f(n)+nf(n)}\dfrac{1}{(2 \pi)^{f(n)/2}\cdot n^{f(n)/2}}\end{displaymath}

\end{lemma}
\begin{proof}
By applying the Stirling approximation \begin{math}k!=\sqrt{2 \pi k }(\frac{k}{e})^k\end{math}  we evaluate:

\begin{math}\\
\begin{aligned}
{n \choose f(n)}^{f(n)}=(\dfrac{n!}{f(n)!(n-f(n))!})^{f(n)} \\ \\
 \approx (\dfrac{\sqrt{2 \pi n}(n/e)^n}{\sqrt{2 \pi f(n)}(f(n)/e)^{f(n)}  
 \sqrt{2 \pi (n-f(n))}((n-f(n))/e)^{(n-f(n))}})^{f(n)} \\ \\
=(\dfrac{n^{n}}{(f(n))^{f(n)}(n-f(n)^{n-f(n)}} \cdot \dfrac{\sqrt{n}}{\sqrt{2 \pi f(n) (n-f(n)}})^{f(n)}
\\ \\ \leq (\dfrac{n}{2 \pi f(n)\cdot f(n)})^{f(n)/2}\cdot \dfrac{n^{nf(n)}}{f(n)^{nf(n)}}
\end{aligned}\\\\
\end{math} 

 where the inequality is due to \begin{math}f(n) \leq n-f(n)\end{math}  for a sufficiently large \begin{math}n\end{math} .
 
\begin{math}\\\\
\begin{aligned}
(\dfrac{n}{2 \pi f(n)\cdot f(n)})^{f(n)/2}\cdot \dfrac{n^{nf(n)}}{f(n)^{nf(n)}}=
(\dfrac{n}{f(n)})^{f(n)+nf(n)}\dfrac{1}{(2 \pi)^{f(n)/2}\cdot n^{f(n)/2}}
\end{aligned}\\\\
\end{math} 
\end{proof}


Our main tool for proving graph polynomials are weakly distinguishing is theorem \ref{theorem classes}, which provides a sufficient condition for a polynomial \begin{math}P_Q\end{math} to be weakly distinguishing. This condition is given in terms of independence and clique functions (see definition \ref{definition ind function}). We will prove  theorem \ref{theorem classes} using the following theorems:

\begin{theorem}[Frieze \cite{frieze1990independence}]
\label{theorem frieze independent}
For a graph \begin{math}G\end{math} , denote by \begin{math}\alpha(G)\end{math}  the size of the largest independent set of vertices in \begin{math}G\end{math} . 
Then for almost all graphs of order \begin{math}n\end{math} , \begin{math}\alpha(G) \approx 4 \log \frac{n}{2}\end{math} 
\end{theorem}
\begin{theorem}[Erd\"os and Bollob\'as \cite{bollobas1976cliques}]
\label{theorem bollobas cliques}
For almost all graphs \begin{math}G\end{math}  of order \begin{math}n\end{math}, \begin{math}\omega(G) \approx \frac{2}{\log 2}\cdot \log n\end{math} 
\end{theorem}

We are now ready to prove theorem \ref{theorem classes}:\\
{\bf Theorem \ref{theorem classes}:} Let \begin{math}Q\end{math}  be a graph property that has an independence or a clique function 
\begin{math}f\end{math}  that satisfies that for all \begin{math}n\in \NN\end{math} ,
\begin{math}f(n)\geq n/a\end{math}  for some fixed \begin{math}a\in \NN\end{math} . Then \begin{math}P_Q\end{math}  is weakly distinguishing.\\

\begin{proof}
Assume \begin{math}f\end{math}  is an independence function.
Set \begin{math}\epsilon=1/10\end{math}  and let \begin{math}\mathcal{C}=\{G:\alpha(G)\leq 4 \log \frac{n}{2}+\epsilon\}\end{math} .
By theorem \ref{theorem frieze independent}, almost all graphs are in \begin{math}\mathcal{C}\end{math}.
Note that if \begin{math}G\in\cC\end{math} , and \begin{math}H\end{math}  is an induced subgraph of \begin{math}G\end{math} with \begin{math}H\in Q\end{math},
then there is an independent set of size \begin{math}\frac{|V(H)|}{a}\end{math} in \begin{math}H\end{math}, and hence in \begin{math}G\end{math}, 
and so \begin{math}|V(H)|\leq 4 \log \frac{n}{2}+\epsilon\end{math}.\\
This implies that
\begin{math}P_Q(G,x)=\sum_{k=1}^{4 \log \frac{|V(G)|}{2}+\epsilon} b_kx^k\end{math}
with \begin{math}0\leq b_k \leq {n \choose k}\end{math} for all \begin{math}k\end{math}, and so
\begin{displaymath}\beta_{P_Q,\cC}(n)\leq {n \choose 4 \log \frac{|V(G)|}{2}+\epsilon}^{4 \log \frac{|V(G)|}{2}+\epsilon}\end{displaymath}
hence by lemmas \ref{lemma 1} and \ref{lemma 2}, \begin{math}P_Q\end{math} is weakly distinguishing.\\

If \begin{math}f\end{math} is a clique function, the proof is similar using theorem \ref{theorem bollobas cliques}.
\end{proof}

\section{Applications of the Method}
\subsection{The Clique and Independence Polynomials}

\begin{definition}
Let \begin{math}G\end{math} be a graph. For \begin{math}i\in \NN \end{math}, denote \begin{math}c_i(G)=|\{A\subseteq V(G):G[A]\cong K_i\}|\end{math}. 
The clique polynomial of \begin{math}G\end{math}, \begin{math}Cl(G,x)\end{math} is defined to be \begin{math}Cl(G,x)=1+\sum_{i=1}^\infty c_i(G)x^i\end{math}. 
Note that this is a graph polynomial, and that the sum in the definition is finite. 
The clique number of \begin{math}G\end{math}, denoted \begin{math}\omega(G)\end{math}, is the degree of the clique polynomial 
(i.e. this is the size of the largest clique subgraph of \begin{math}G\end{math}).
\end{definition}

\begin{definition}
Let \begin{math}G\end{math} be a graph. For \begin{math}i\in \NN\end{math}, denote \begin{math}s_i(G)=|\{A\subseteq V(G):G[A]\cong I_i\}|\end{math}. 
The independence polynomial of \begin{math}G\end{math}, \begin{math}Ind(G,x)\end{math} is defined to be \begin{math}Ind(G,x)=1+\sum_{i=1}^\infty s_i(G)x^i\end{math}. 
Note that this is a graph polynomial, and that the sum in the definition is finite. 
The independence number of \begin{math}G\end{math}, denoted \begin{math}\alpha(G)\end{math}, is the degree of the independence polynomial 
(i.e. this is the size of the largest independent set in \begin{math}G\end{math}).
\end{definition}
Theorem \ref{theorem clique ind} is now a direct corollary of theorem \ref{theorem classes}:\\
{\bf Theorem \ref{theorem clique ind}}: Almost all graphs \begin{math}G\end{math} have an \begin{math}Ind\end{math}-mate and a \begin{math}Cl\end{math}-mate.\\
\begin{proof}
For the independence polynomial, note that \begin{math}Ind(G,x)=P_Q(G;x)\end{math}, were \begin{math}Q\end{math} is the property consisting of edgeless graphs. Note that the identity function on \begin{math}\NN \end{math} is an independence function for \begin{math}Q\end{math}, and clearly it satisfies the condition in theorem \ref{theorem classes} for \begin{math}a=1\end{math}. Hence the independence polynomial is weakly distinguishing.\\
Similarly, for the clique polynomial note that \begin{math}Cl(G,x)=P_Q(G;x)\end{math}, were \begin{math}Q\end{math} is the property of complete graphs. Note that the identity function on \begin{math}\NN \end{math} is a clique function for \begin{math}Q\end{math}, and clearly it satisfies the condition in theorem \ref{theorem classes} for \begin{math}a=1\end{math}. Hence the clique polynomial is weakly distinguishing.
\end{proof}

\subsection{Generating Functions}
Theorem \ref{theorem classes} can be applied to many graph classes to produce weakly distinguishing graph polynomials.
Of particular interest are \begin{math}k\end{math}-degenerate classes and amongst them classes of bounded treewidth.

For a graph \begin{math}G\end{math}, and \begin{math}v\in V(G)\end{math} denote by 
\begin{math}N_G(v)\end{math} the closed neighbourhood of \begin{math}v\end{math} in \begin{math}G\end{math}, i.e. 
\begin{math}N_G(v)=\{v\}\cup \{u\in V(G): \{v,u\}\in E(G)\}\end{math}
\begin{definition}
For \begin{math}k\in \NN\end{math}, a graph \begin{math}G\end{math} is said to be \begin{math}k\end{math}-degenerate if every induced subgraph of \begin{math}G\end{math} has a vertex of degree at most \begin{math}k\end{math}. 
\end{definition}
The following propositions \ref{theorem 10}, \ref{theorem 12} and lemma \ref{lemma 11} are well known results about degenerate graphs and treewidth. For completeness, we include their proofs:
\begin{proposition}
\label{theorem 10}
A graph \begin{math}G=(V,E)\end{math} is \begin{math}k\end{math} degenerate if and only if there is a enumeration \begin{math}\{v_1,v_2,...,v_n\}=V\end{math} such that for all \begin{math}1\leq i \leq n \end{math}the degree of \begin{math}v_i\end{math} 
in the subgraph of \begin{math}G\end{math} induced by \begin{math}V-\{v_1,v_2,...,v_{i-1}\}\end{math} is at most \begin{math}k\end{math}.
\end{proposition}
\begin{proof}
Let \begin{math}G=(V,E)\end{math} be a \begin{math}k\end{math} degenerate graph. From the definition, there is a vertex \begin{math}v\in V\end{math} with degree at most \begin{math}k\end{math}. Denote this vertex by \begin{math}v_1\end{math}. 
Define \begin{math}v_i\end{math} inductively: from the definition, the subgraph induced by \begin{math}V-\{v_1,...,v_{i-1}\}\end{math} has a vertex with degree at most \begin{math}k\end{math}. 
Define \begin{math}v_i\end{math} to be this vertex. Clearly, the enumeration \begin{math}\{v_1,...,v_n\}=V\end{math} has the desired property.\\
Conversely, let \begin{math}\{v_1,...,v_n\}\end{math} an enumeration as in the theorem, and let \begin{math}H\end{math} be a subgraph of \begin{math}G\end{math} induced by \begin{math}U\subseteq V\end{math}.
Denote \begin{math}u=v_i\end{math} the vertex in \begin{math}H\end{math} who's index in the enumeration is the smallest. Note that the degree of \begin{math}u\end{math} in \begin{math}G[U \cup \{v_j|j \geq i\}]\end{math} 
is at most \begin{math}k\end{math}, and \begin{math}H\end{math} is a subgraph of \begin{math}G[U \cup \{v_j|j \geq i\}]\end{math}, hence the degree of \begin{math}u\end{math} in \begin{math}H\end{math} is at most \begin{math}k\end{math}.
So \begin{math}G\end{math} is \begin{math}k\end{math} degenerate, as required.
\end{proof}

\begin{lemma}
\label{lemma 11}
A graph with treewidth at most \begin{math}k\end{math} has a vertex with degree at most \begin{math}k\end{math}.
\end{lemma}
\begin{proof}
Let \begin{math}G\end{math} be a graph, and \begin{math}(T,X)\end{math} a tree decomposition of \begin{math}G\end{math} with width \begin{math}k\end{math}. 
Note that \begin{math}T\end{math} has a leaf, and there is a vertex \begin{math}v\end{math} in the bag corresponding to this 
leaf that is not in the bag corresponding to its neighbour. 
Thus every neighbour of \begin{math}v\end{math} in \begin{math}G\end{math} is in the same bag. But the bag is of size at most \begin{math}k+1\end{math}, so \begin{math}v\end{math} is of degree at most \begin{math}k\end{math}.
\end{proof}

\begin{proposition}
\label{theorem 12}
A graph \begin{math}G\end{math} with treewidth at most \begin{math}k\end{math} is \begin{math}k\end{math} degenerate.
\end{proposition}
\begin{proof}
Let \begin{math}H\end{math} be an induced subgraph of \begin{math}G\end{math}. If \begin{math}H=G\end{math}, then \begin{math}H\end{math} has a vertex of degree at most \begin{math}k\end{math} by the previous lemma.
 If \begin{math}H\end{math} is a proper subgraph, note that \begin{math}H\end{math} has treewidth at most \begin{math}k\end{math}, so \begin{math}H\end{math} has a vertex of degree at most \begin{math}k\end{math}. So \begin{math}G\end{math} is \begin{math}k\end{math} degenerate.
\end{proof}\\
The following proposition shows that a \begin{math}k\end{math} degenerate graph has a large independent set:
\begin{proposition}
Every \begin{math}k\end{math} degenerate graph \begin{math}G\end{math} has an independent set of size \begin{math}{\left\lceil \frac{|V|}{k+1} \right\rceil}\end{math}.
\end{proposition}
\begin{proof}
Let \begin{math}G=(V,E)\end{math} be a \begin{math}k\end{math} degenerate graph, and \begin{math}\{v_1,...,v_n\}\end{math} be an enumeration as in proposition \ref{theorem 10}. Let $I_0=\emptyset$, and $H_0=G$. We construct an independent set inductively as follows. There exists $1\leq l \leq n$ and an increasing sequence $i_1,i_2,...,i_l$ in $\{1,2,...,n\}$  with $i_1=1$ such that for all $1\leq j \leq l$
\begin{align*}
I_j=I_{j-1}\cup \{v_{i_j}\}\\
H_j=G[V(H_{j-1})-N_G(v_j)]
\end{align*}
with $I_l$ an independent set in $G$ and $l\geq {\left\lceil \frac{n}{k+1} \right\rceil}$.

Indeed, $I_1=\{v_1\}$ and $H_1=G[V(G)-N_G(v_1)]$. Clearly, $I_1$ is an independent set, and note that $|V(H_1)|\geq n-(k+1)$ and no vertex in $H_1$ is a neighbour of the vertex in $I_1$. Now, given an independent set $I_j$ and an induced subgraph $H_j$ of $G$ such that $|V(H_j)|\geq n-j(k+1)$ and no vertex of $H_j$ is a neighbour of a vertex in $I_j$, select $v_{i_{j+1}}\in V(H_j)$ with minimal index. Now $I_{j+1}$ is an independent set, no vertex of $H_{j+1}$ is a neighbour of a vertex in $I_{j+1}$, and since $deg_{H_j}(v_{i_{j+1}})\leq k$,\\ $|V(H_{j+1})|\geq |V(H_j)|-(k+1)\geq n-(j+1)(k+1)$.\\

The induction stops when no more vertices can be selected, i.e. when $V(H_l)=\emptyset$. From the induction, we have that $0=|V(H_l)|\geq n-l(k+1)$ and hence $l\geq{\left\lceil \frac{n}{k+1} \right\rceil}$ as required.
\end{proof}\\
Combining this proposition with theorem \ref{theorem classes}, we can show that many non trivial graph polynomials are weakly distinguishing:
\begin{corollary}
Fix \begin{math}k\in \NN\end{math}, and let \begin{math}Q\end{math} be a class of \begin{math}k\end{math} degenerate graphs. Then \begin{math}P_Q\end{math} is weakly distinguishing. 
\end{corollary}
\begin{corollary}
Fix \begin{math}k\in \NN\end{math}, and let \begin{math}Q\end{math} be a class of graphs with treewidth at most \begin{math}k\end{math}. Then \begin{math}P_Q\end{math} is weakly distinguishing.
\end{corollary}

\section{The Harmonious and k-Harmonious Polynomials}

\begin{definition}
For a graph \begin{math}G\end{math}, a harmonious colouring in \begin{math}k\end{math} colours is a function \begin{math}f:V(G) \goesto [k]\end{math} 
such that \begin{math}f\end{math} is a proper colouring, and for all \begin{math}i,j\in [k]\end{math}, \begin{math}G[f^{-1}(i)\cup f^{-1}(j)]\end{math} has at most one edge. 
Denote by \begin{math}\chi_{harm}(G,\lambda)\end{math} the number of \begin{math}\lambda\end{math} harmonious colourings of \begin{math}G\end{math}.
 Then \begin{math}\chi_{harm}\end{math} is a polynomial in \begin{math}\lambda\end{math}, as shown in \cite{makowsky2006polynomial} 
 and \cite{godlin2008evaluations}. \begin{math}\chi_{harm}\end{math} is called the harmonious polynomial.
\end{definition}

For more on the harmonious polynomial, see \cite{drgas2017harmonious}. Theorem \ref{theorem harmonious} was observed without proof in \cite{drgas2017harmonious}.\\
{\bf Theorem \ref{theorem harmonious}:} Almost all graphs \begin{math}G\end{math} have a \begin{math}\chi_{harm}\end{math}-mate.\\

\begin{proof}
Let \begin{math}\mathcal{C}\end{math} be the class of graphs \begin{math}G\end{math} that have the property that for every two vertices \begin{math}v,u\in V(G)\end{math}, 
there is a vertex \begin{math}w\in V(G)\end{math} such that \begin{math}w\end{math} is a neighbour of both \begin{math}v\end{math} and \begin{math}u\end{math}. 
This property is one of Gaifman's extension axioms, and hence from Fagin's proof of the 0/1-law for first order logic, 
almost all graphs are in \begin{math}\mathcal{C}\end{math}(see \cite{fagin1976probabilities} for details).\\

Note that any harmonious colouring of a graph \begin{math}G\in \mathcal{C}\end{math} of order \begin{math}n\end{math} 
has to assign a different colour to each vertex of \begin{math}G\end{math}, and so for \begin{math}\lambda \in \NN\end{math} the evaluation 
of the harmonious polynomial of \begin{math}G\end{math} at \begin{math}\lambda\end{math} is \begin{math}\chi_{harm}(G,\lambda)=\lambda(\lambda-1)(\lambda-2)...(\lambda-n+1)\end{math}. 
Since this is true for every \begin{math}\lambda \in \NN\end{math}, by interpolation it is true for every \begin{math}\lambda\in \RR\end{math}, 
and so every two graphs in \begin{math}\mathcal{C}\end{math} of the same order have the same harmonious polynomial. 
Hence, there is an \begin{math}n_0\end{math} such that all graphs in \begin{math}\mathcal{C}\end{math} of order greater than \begin{math}n_0\end{math} have an \begin{math}\chi_{harm}\end{math}-mate. 
Thus the harmonious polynomial is weakly distinguishing.
\end{proof}

This result can be easily generalised.
\begin{definition}
For a fixed \begin{math}k \in \NN\end{math} and a graph \begin{math}G\end{math}, 
we say that a proper colouring of \begin{math}G\end{math} with \begin{math}\lambda\end{math} colours \begin{math}f:V(G)\goesto [\lambda]\end{math}  is \begin{math}k\end{math}-harmonious 
if for every \begin{math}S\subseteq [\lambda]\end{math} such that \begin{math}|S|=k\end{math}, \begin{math}S\end{math} appears as the colour set of a clique in the graph at most once, 
i.e. if \begin{math}\{v_1,v_2,...,v_k\},\{u_1,...,u_k\}\subseteq V(G)\end{math} induce complete graphs of size \begin{math}k\end{math} and 
\begin{math}f( \{v_1,v_2,...,v_k\})=f(\{u_1,...,u_k\})\end{math}, then \begin{math}\{v_1,v_2,...,v_k\}=\{u_1,...,u_k\}\end{math}.
\end{definition}

For \begin{math}\lambda \in \NN\end{math} define \begin{math}h_k(G,\lambda)=|\{f:V(G)\goesto [\lambda]:f\end{math} is proper and \begin{math}k\end{math}-harmonious\begin{math}\}|\end{math}. 
\begin{math}h_k(G,\lambda)\end{math} is a polynomial in \begin{math}\lambda\end{math} (again, by  \cite{makowsky2006polynomial}). 
We will prove  that \begin{math}h_k\end{math} is weakly distinguishing.\\

We start with a lemma:

\begin{lemma}
Let \begin{math}\mathcal{C}\end{math} be the class of graphs  \begin{math}G\end{math} with the property that for every two vertices \begin{math}v,u \in V(G)\end{math} 
there are vertices \begin{math}w_1,...,w_k\in V(G)\end{math} such that \begin{math}\{u,w_1,...,w_k\}\end{math} and \begin{math}\{v,w_1,...,w_k\}\end{math} induce a complete graph. 
Then almost all graphs are in \begin{math}\mathcal{C}\end{math}.
\end{lemma}

For convenience, we restate and prove the lemma in probabilistic language:

\begin{lemma}
Fix \begin{math}p\in (0,1)\end{math} and let \begin{math}G\in \mathcal{G}(n,p)\end{math} (i.e. \begin{math}G\end{math} is a graph with \begin{math}n\end{math} 
vertices and every edge is in the graph with probability \begin{math}p\end{math}, independently of the others). 
Then \begin{math}\lim_{n\goesto \infty}\PP(G\in \mathcal{C})=1\end{math}
\end{lemma}

\begin{proof}
For a graph $G$, denote by $Y_k(G)$ the number of $k$ cliques in $G$.
For fixed \begin{math}u,v\in V(G)\end{math}, denote by \begin{math}G_{u,v}\end{math} the subgraph of \begin{math}G\end{math} induced by the common neighbours of \begin{math}v\end{math} and \begin{math}u\end{math}. 
Note that \begin{math}\EE(|V(G_{u,v})|)=(n-2)p^2\end{math}, so from the multiplicative Chernoff bound,
\begin{align*}
\PP[|V(G_{u,v})|\leq \frac{1}{9}(n-2)p^2]\leq \left(\frac{e^{-9/10}}{(1/9)^{1/9}}\right)^{(n-2)p^2}<e^{-\frac{1}{200}(n-2)p^2}
\end{align*} 

Hence, from the union bound 

\begin{displaymath}\PP[\exists u,v\in V(G) s.t. |V(G_{u,v})|\leq \frac{1}{9}(n-2)p^2]\leq {n \choose 2}e^{-\frac{1}{200}(n-2)p^2}\end{displaymath}

Denote this number \begin{math}r_p(n)\end{math}, and note that \begin{math}\lim_{n\goesto \infty} r_p(n)=0\end{math}.\\

Next, assume that \begin{math}\forall u,v\in V(G)\end{math}, \begin{math}|V(G_{u,v})|>\frac{1}{9}(n-2)p^2\end{math}. Fix \begin{math}u,v\in V(G)\end{math}. Then \begin{math}G_{u,v}\end{math} is a random graph with more than \begin{math}\frac{1}{9}(n-2)p^2\end{math} vertices, and hence \begin{math}\PP[Y_k(G_{u,v})< 1]\leq \PP[Y_k(G')< 1]\end{math} were \begin{math}G'\in\mathcal{G}(\frac{1}{9}(n-2)p^2,p)\end{math}.
From theorem 2 in \cite{bollobas1988chromatic}, we have:

\begin{displaymath}\PP[Y_k(G')<1]\leq \PP \left[ Y_k(G')\leq\frac{9}{10}\left(\frac{1}{9}(n-2)p^2\right)^{3/2}\right]
\leq \exp\left[ -(1/100  +  \alpha \sqrt{\frac{1}{9}(n-2)p^2})\right]\end{displaymath}

for some constant \begin{math}\alpha>0\end{math}. Hence, if we denote \begin{math}\mathcal{A}=\{G:\exists u,v\in V(G) s.t. |V(G_{u,v})|\leq\frac{1}{9}(n-2)p^2\}\end{math}, from the union bound we have:
\begin{displaymath}
\PP \left[ \exists u,v\in V(G)s.t. Y_k(G_{u,v})<1|G\not \in \mathcal{A}\right] \leq {n \choose 2}\exp\left[ -(1/100  +  \alpha \sqrt{\frac{1}{9}(n-2)p^2})\right]
\end{displaymath}

Denote the right side of this inequality by \begin{math}r'(n)\end{math}, and note that \begin{math}r'(n)\goesto 0\end{math} as \begin{math}n \goesto \infty\end{math}.

To conclude, we have that
\begin{align*}
\PP[G \not \in \mathcal{C}]=\PP[G\in \mathcal{A}\cap \mathcal{C}^c]+\PP[G\in\mathcal{A}^c \cap\mathcal{C}^c]\leq\\ \leq\PP[G\in \mathcal{A}]+\PP[G\in \mathcal{A}^c]\PP[G\in \mathcal{C}^c|G\in \mathcal{A}^c]=r_p(n)+(1-r_p(n))r'_p(n)
\end{align*}

and both these terms tend to 0 as \begin{math}n\end{math} tends to infinity.
\end{proof}

We can now prove:

\begin{theorem}
For a fixed \begin{math}k\in\NN\end{math}, \begin{math}h_k\end{math} is weakly distinguishing.
\end{theorem}

\begin{proof}
Let \begin{math}\mathcal{C}\end{math} be the same as in the lemma. Similarly to the previous theorem, note that if \begin{math}G\in\mathcal{C}\end{math}, 
then for \begin{math}\lambda\in \NN\end{math} any \begin{math}k\end{math}-harmonious colouring \begin{math}f:V(G) \goesto [\lambda]\end{math} of \begin{math}G\end{math} must assign 
a different colour to every vertex (otherwise, if \begin{math}v,u \in V(G)\end{math} are such that 
\begin{math}f(v)=f(u)\end{math}, and \begin{math}w_1,...,w_k\in V(G)\end{math} are such that \begin{math}\{u,w_1,...,w_k\}\end{math} and \begin{math}\{v,w_1,...,w_k\}\end{math} 
induce a complete graph in \begin{math}G\end{math}, then \begin{math}f(\{u,w_1,...,w_k\})=f(\{v,w_1,...,w_k\})\end{math}), 
and thus \begin{math}h_k(G,\lambda)=\lambda(\lambda-1)(\lambda-2)...(\lambda-n+1)\end{math}. 
By the same reasoning as in the proof of theorem \ref{theorem harmonious}, \begin{math}h_k\end{math} is weakly distinguishing.
\end{proof}

\section{Conclusion}
We have shown for many graph properties \begin{math}Q\end{math} that the polynomials \begin{math}P_Q\end{math} are weakly distinguishing, including the well studied clique and independence polynomials. We have also shown that the harmonious and k-harmonious polynomials are weakly distinguishing.\\
Our results relied on the fact that the number of polynomials that can be the \begin{math}P_Q\end{math} polynomial of a graph is small, and on the fact that almost all graphs have properties that imply their harmonious and k-harmonious polynomials are trivial. This does not seem to be the case for the Tutte and the chromatic polynomials, so the original question of whether they, as well as the characteristic and matching 
polynomials, are weakly distinguishing remains open.

\begin{problem}
Find a graph property $Q$ such that the fraction of $P_Q$ unique graphs is strictly positive. 
\end{problem}
\begin{problem}
Let $P$ be a graph polynomial, $G,H$ two graphs, and write $G\sim_P H$ if $P(G)=P(H)$. What can be said about the sizes of the equivalence classes of $\sim_P$?
\end{problem}
\bibliography{Quotes}
\bibliographystyle{plain}

\end{document}